\documentclass[12pt]{article}
\usepackage{amssymb}
\usepackage{amsthm}
\usepackage{amsmath}
\usepackage{graphicx}
\usepackage{enumerate}
\usepackage{pict2e}
\usepackage{framed}

\newtheorem{theorem}{Theorem}[section]

\newtheorem{lemma}[theorem]{Lemma}

\newtheorem{example}[theorem]{Example}
\newtheorem{corollary}[theorem]{Corollary}
\title{Properties of the symmetric difference in lattices with complementation}
\author{V\'aclav~Cenker, Ivan~Chajda and Helmut~L\"anger}
\date{}

\begin{document}

\footnotetext{Support of the research of the first author by the Czech Science Foundation (GA\v CR), project 24-14386L, support of the research of the first two authors by IGA, project P\v rF~2025~008, support of the research of the second author by the Czech Science Foundation (GA\v CR), project 25-20013L, and support of the research of the third author by the Austrian Science Fund (FWF), project 10.55776/PIN5424624, is gratefully acknowledged.}

\maketitle
	
\begin{abstract}
The symmetric difference in Boolean lattices can be defined in two different but equivalent forms. However, it can be introduced also in every bounded lattice with complementation where these two forms need not coincide. We study lattices with complementation and the variety of such lattices where these two expressions coincide and point out explicitly some interesting subvarieties. Using a result of J.~Berman we estimate the size of free algebras in these subvarieties. It is well-known that the symmetric difference is associative in every Boolean lattice. We prove that it is just the property of Boolean lattices, namely the symmetric difference in a lattice with complementation is associative if and only if this lattice is Boolean. Similarly, we prove that a lattice with complementation is Boolean if and only if the symmetric difference satisfies a certain simple identity in two variables. We also characterize lattices with a unary operation satisfying De Morgan's laws.
\end{abstract}
	
	

\section{Introduction}
The aim of this paper is to show several important properties of the symmetric difference in complemented lattices where the complementation is a unary operation. Because the symmetric difference can be defined in two different ways which coincide in Boolean lattices, we are interested in lattices which need not be distributive but these two forms of symmetric difference coincide. It is well-known that in a Boolean algebra the symmetric difference is associative. We show that also, conversely, associativity of the symmetric difference yields distributivity of the complemented lattice in question.

We start with definitions of several basic concepts.

Let $\mathbf L=(L,\vee,\wedge,0,1)$ be a bounded lattice and $a\in L$. An element $b$ of $L$ is called a {\em complement} of $a$ if $a\vee b=1$ and $a\wedge b=0$. The {\em lattice} $\mathbf L$ is called {\em complemented} if every element of $L$ has a complement which need not be unique. Let $\,'$ be a unary operation on $L$. Then $\,'$ is called 
\begin{enumerate}[(i)]
\item {\em antitone} if, for all $x,y\in L$, $x\le y$ implies $y'\le x'$,
\item an {\em involution} if $x''=x$ for all $x\in L$,
\item a {\em complementation} if, for every $x\in L$, $x\vee x'=1$ and $x\wedge x'=0$.
\end{enumerate}
In (iii), $x'$ is a complement of $x$ for every $x\in L$. The algebra $\mathbf L=(L,\vee,\wedge,{}',0,1)$ where $\,'$ is a complementation will be called a {\em lattice with complementation}. If, moreover, $\,'$ is an antitone involution then $\mathbf L$ is called an {\em ortholattice}.

At first we prove an auxiliary result using identites similar to those studied in \cite{CLP} and \cite{CP}.

\begin{lemma}\label{lem2}
Let $\mathbf L=(L,\vee,\wedge,{}',0,1)$ be a lattice with complementation satisfying one of the following identities:
\begin{enumerate}[{\rm(i)}]
\item $x\wedge y\approx x\wedge(x'\vee y)$,
\item $x\vee y\approx x\vee(x'\wedge y)$.
\end{enumerate}
Then $\mathbf L$ satisfies the identity $x''\approx x$.	
\end{lemma}

\begin{proof}
Assume (i). Then
\begin{equation}\label{equ10}
x\wedge x''\approx x\wedge(x'\vee x'')\approx x\wedge1\approx x.
\end{equation}
Hence $x'\approx x'\wedge x''$ and
\[
x'''\vee x'\approx x'''\vee(x'\wedge x''')\approx x'''
\]
whence
\[
x\vee x'''\approx x\vee x'''\vee x'\approx1.
\]
Using this, (\ref{equ10}) and (i) we obtain
\[
x\approx x\wedge x''\approx x''\wedge x\approx x''\wedge(x'''\vee x)\approx x''\wedge1\approx x''.
\]
By dualizing the preceding proof we see that (ii) implies $x''\approx x$.		
\end{proof}

It is well-known (see e.g.\ \cite D) that in a distributive bounded lattice every element has at most one complement. A distributive lattice with complementation is an ortholattice which is called a {\em Boolean lattice}. Hence in Boolean lattices $\mathbf L$, for every $x\in L$, $x'$ is the unique complement of $x$.

It is evident that if $(L,\vee,\wedge,{}',0,1)$ is a lattice with complementation and $a\in L$ then $a=0$ if and only if $a'=1$ and, conversely, $a=1$ if and only if $a'=0$.

In every bounded lattice $\mathbf L=(L,\vee,\wedge,{}',0,1)$ with a unary operation $\,'$ one can introduce two so-called {\em symmetric differences}, i.e.\ the term operations
\begin{align*}
x+_1y & :=(x'\wedge y)\vee(x\wedge y'), \\
x+_2y & :=(x\vee y)\wedge(x'\vee y').
\end{align*}
which coincide if $\mathbf L$ is a Boolean lattice, see e.g.\ \cite{B79}. The question is whether this holds only in Boolean lattices. Hence, in what follows, we are interested in lattices with complementation satisfying the identity
\begin{equation}\label{equ1}
(x'\wedge y)\vee(x\wedge y')\approx(x\vee y)\wedge(x'\vee y').
\end{equation}
We call this identity the {\em coincidence identity}. Let us note that varieties of lattices with complementation where the complementation satisfies certain adjointness properties were already studied by the authors in \cite{CCL}.

\begin{example}\label{ex4}
Consider the modular lattice $\mathbf M_3=(M_3,\vee,\wedge)$ depicted in Fig.~1.

\vspace*{-4mm}

\begin{center}
\setlength{\unitlength}{7mm}
\begin{picture}(6,6)
\put(3,1){\circle*{.3}}
\put(1,3){\circle*{.3}}
\put(3,3){\circle*{.3}}
\put(5,3){\circle*{.3}}
\put(3,5){\circle*{.3}}
\put(3,1){\line(-1,1)2}
\put(3,1){\line(0,1)4}
\put(3,1){\line(1,1)2}
\put(3,5){\line(-1,-1)2}
\put(3,5){\line(1,-1)2}
\put(2.85,.3){$0$}
\put(.35,2.85){$a$}
\put(3.4,2.85){$b$}
\put(5.4,2.85){$c$}
\put(2.85,5.4){$1$}
\put(2.2,-.75){{\rm Fig.~1}}
\put(.4,-1.75){{\rm Modular lattice $\mathbf M_3$}}
\end{picture}
\end{center}

\vspace*{10mm}

There are exactly six possibilities how to define a complementation $\,'$ on $\mathbf M_3$. 

Each of these complementations is antitone, but none is an involution. One such complementation is given by
\[
\begin{array}{l|ccccc}
x  & 0 & a & b & c & 1 \\
\hline
x' & 1 & b & c & a & 0
\end{array}\]
We have
\begin{align*}
a+_1b & =(a'\wedge b)\vee(a\wedge b')=(b\wedge b)\vee(a\wedge c)=b\vee0=b\ne1=1\wedge(b\vee c)= \\
      & =(a\vee b)\wedge(a'\vee b')=a+_2b,
\end{align*}
i.e., the lattice $(M_3,\vee,\wedge,{}',0,1)$ with complementation does not satisfy identity {\rm(\ref{equ1})} {\rm(}cf.\ Corollary~\ref{cor1}{\rm)}.
\end{example}

For lattices with complementation satisfying identity (\ref{equ1}) we can show the following result.

\begin{lemma}
Let $\mathbf L=(L,\vee,\wedge,{}',0,1)$ be a lattice with complementation satisfying identity {\rm(\ref{equ1})} and let $a\in L$. Then the following hold:
\begin{enumerate}[{\rm(i)}]
\item $\mathbf L$ satisfies the identity $x'\vee(x\wedge x'')\approx1$,
\item $a\wedge a''\le a'$ if and only if $a=0$.
\end{enumerate}	
\end{lemma}

\begin{proof}
\
\begin{enumerate}[(i)]
\item Identity (\ref{equ1}) implies $x'\vee(x\wedge x'')\approx(x'\wedge x')\vee(x\wedge x'')\approx(x\vee x')\wedge(x'\vee x'')\approx1\wedge1\approx1$.
\item Because of (i) the following are equivalent $a\wedge a''\le a'$, $a'\vee(a\wedge a'')=a'$, $a'=1$, $a=0$.
\end{enumerate}
\end{proof}

Recall that the identities $(x\vee y)'\approx x'\wedge y'$ and $(x\wedge y)'\approx x'\vee y'$ for lattices $(L,\vee,\wedge,{}')$ with a unary operation $\,'$ are called {\em De Morgan's laws}.

\section{Lattices and varieties of lattices satisfying the coincidence identity for symmetric differences}

The following lemma presents two identities holding in any ortholattice satisfying identity (\ref{equ1}).

\begin{lemma}
Let $(L,\vee,\wedge,{}',0,1)$ be an ortholattice satisfying identity {\rm(\ref{equ1})}. Then it satisfies the identities
\begin{align*}
(x\wedge y)\vee(x\wedge y')\vee(x'\wedge y)\vee(x'\wedge y') & \approx1, \\
  (x\vee y)\wedge(x\vee y')\wedge(x'\vee y)\wedge(x'\vee y') & \approx0.
\end{align*}
\end{lemma}

\begin{proof}
We have
\[
(x\wedge y)\vee(x\wedge y')\vee(x'\wedge y)\vee(x'\wedge y')\approx(x+_1y)\vee(x+_2y)'\approx(x+_1y)\vee(x+_1y)'\approx1.	
\]
The second identity is dual to the first one.
\end{proof}

In the previous example, every of the elements $a,b,c$ has two incomparable complements which, moreover, are complements of each other. We can show that there exist lattices satisfying identity (\ref{equ1}) having an element with two incomparable complements that are not complements of each other, see the following example.

\begin{example}
In Fig.~2 and 3 ortholattices satisfying identity {\rm(\ref{equ1})} are visualized:
	
\vspace*{-4mm}
	
\begin{center}
\setlength{\unitlength}{7mm}
\begin{picture}(16,10)
\put(2,1){\circle*{.3}}
\put(0,3){\circle*{.3}}
\put(4,3){\circle*{.3}}
\put(-1,5){\circle*{.3}}
\put(1,5){\circle*{.3}}
\put(3,5){\circle*{.3}}
\put(5,5){\circle*{.3}}
\put(0,7){\circle*{.3}}
\put(4,7){\circle*{.3}}
\put(2,9){\circle*{.3}}
\put(13,1){\circle*{.3}}
\put(11,3){\circle*{.3}}
\put(13,3){\circle*{.3}}
\put(15,3){\circle*{.3}}
\put(11,5){\circle*{.3}}
\put(13,5){\circle*{.3}}
\put(15,5){\circle*{.3}}
\put(13,7){\circle*{.3}}
\put(14,2){\circle*{.3}}
\put(12,6){\circle*{.3}}
\put(2,1){\line(-1,1)2}
\put(2,1){\line(1,1)2}
\put(0,3){\line(-1,2)1}
\put(0,3){\line(1,2)1}
\put(4,3){\line(-1,2)1}
\put(4,3){\line(1,2)1}
\put(0,7){\line(-1,-2)1}
\put(0,7){\line(1,-2)1}
\put(4,7){\line(-1,-2)1}
\put(4,7){\line(1,-2)1}
\put(2,9){\line(-1,-1)2}
\put(2,9){\line(1,-1)2}
\put(13,1){\line(-1,1)2}
\put(13,1){\line(0,1)2}
\put(13,1){\line(1,1)2}
\put(13,3){\line(-1,1)2}
\put(13,3){\line(1,1)2}
\put(13,5){\line(-1,-1)2}
\put(13,5){\line(1,-1)2}
\put(13,7){\line(-1,-1)2}
\put(13,7){\line(0,-1)2}
\put(13,7){\line(1,-1)2}
\put(11,3){\line(0,1)2}
\put(15,3){\line(0,1)2}
\put(1.85,.3){$0$}
\put(-.65,2.85){$a$}
\put(4.35,2.85){$b$}
\put(-1.65,4.85){$c$}
\put(1.35,4.85){$d$}
\put(2.35,4.85){$d'$}
\put(5.35,4.85){$c'$}
\put(-.65,6.85){$b'$}
\put(4.35,6.85){$a'$}
\put(1.85,9.4){$1$}
\put(12.85,.3){$0$}
\put(10.35,2.85){$b$}
\put(15.35,2.85){$d$}
\put(13.35,2.85){$c$}
\put(14.35,1.85){$a$}
\put(10.35,4.85){$d'$}
\put(15.35,4.85){$b'$}
\put(13.35,4.85){$c'$}
\put(12.85,7.4){$1$}
\put(11.35,5.85){$a'$}
\put(1.2,-.75){{\rm Fig.~2}}
\put(-2.2,-1.75){{\rm Ortholattice satisfying identity (\ref{equ1})}}
\put(12.2,-.75){{\rm Fig.~3}}
\put(9,-1.75){{\rm Ortholattice satisfying identity (\ref{equ1})}}
\end{picture}
\end{center}
	
\vspace*{10mm}

As one can see, in the ortholattice depicted in Fig.~2 the elements $c'$ and $d'$ are incomparable complements of $c$, but $c'$ and $d'$ are not complements of each other whereas in the ortholattice visualized in Fig.~3 the elements $a'$ and $d'$ are comparable complements of $a$. Moreover, the first ortholattice is not subdirectly irreducible while the second is. The congruence monolith of the second ortholattice contains $\{a,d\}$ and $\{d',a'\}$ as non-trivial classes.
\end{example}

However, the situation shown in Example~\ref{ex4} can be generalized as follows.

\begin{theorem}
Let $(L,\vee,\wedge,{}',0,1)$ be a non-trivial lattice with complementation satisfying identity {\rm(\ref{equ1})} and let $a\in L$. Then there does not exist some $b\in L$ being a complement of $a$ and $a'$.
\end{theorem}

\begin{proof}
If there would exist some $b\in L$ being a complement of $a$ and $a'$ then using (\ref{equ1}) we would obtain
\begin{align*}
a & =a\vee0\vee(a\wedge b')=a\vee(a'\wedge b)\vee(a\wedge b')=a\vee\big((a\vee b)\wedge(a'\vee b')\big)= \\
  & =a\vee\big(1\wedge(a'\vee b')\big)=a\vee a'\vee b'=1
\end{align*}
and hence
\[
0=b\wedge a=b\wedge1=b=b\vee0=b\vee(a'\wedge a)=b\vee(a'\wedge1)=b\vee a'=1,
\]
a contradiction.
\end{proof}

\begin{corollary}\label{cor1}
A lattice $(L,\vee,\wedge,{}',0,1)$ with complementation satisfying identity {\rm(\ref{equ1})} cannot contain a subalgebra isomorphic to $(M_3,\vee,\wedge,{}',0,1)$.
\end{corollary}
\label{key}
On the other hand, we can show that if $x$ and $y$ or $x$ and $y'$ are comparable for all $x,y\in L$ then the ortholattice $\mathbf L$ satisfies identity (\ref{equ1}).

\begin{theorem}\label{th1}
Let $\mathbf L=(L,\vee,\wedge,{}',0,1)$ be an ortholattice having the property that for all $x,y\in L$, either $x$ and $y$ or $x$ and $y'$ are comparable with each other. Then $\mathbf L$ satisfies identity {\rm(\ref{equ1})}.
\end{theorem}

\begin{proof}
Let $a,b\in L$. \\
If $a\le b$ then $b'\le a'$ and hence $a\wedge b'\le b\wedge a'$ and
\[
a+_1b=(a'\wedge b)\vee(a\wedge b')=(b\wedge a')\vee(a\wedge b')=b\wedge a'=(a\vee b)\wedge(a'\vee b')=a+_2b.
\]
If $b\le a$ then $a'\le b'$ and hence $b\wedge a'\le a\wedge b'$ and
\[
a+_1b=(a'\wedge b)\vee(a\wedge b')=(b\wedge a')\vee(a\wedge b')=a\wedge b'=(a\vee b)\wedge(a'\vee b')=a+_2b.
\]
If $a\le b'$ then $b\le a'$ and hence $b\vee a\le a'\vee b'$ and
\[
a+_1b=(a'\wedge b)\vee(a\wedge b')=b\vee a=(b\vee a)\wedge(a'\vee b')=(a\vee b)\wedge(a'\vee b')=a+_2b.
\]
If $b'\le a$ then $a'\le b$ and hence $a'\vee b'\le b\vee a$ and
\[
a+_1b=(a'\wedge b)\vee(a\wedge b')=a'\vee b'=(b\vee a)\wedge(a'\vee b')=(a\vee b)\wedge(a'\vee b')=a+_2b.
\]
\end{proof}

Recall that the {\em horizontal sum} of a family $(C_i,\le_i,0,1),i\in I,$ of bounded chains with $C_i\cap C_j=\{0,1\}$ for all $i,j\in I$ with $i\ne j$ is the bounded lattice
\[
\left(\bigcup_{i\in I}C_i,\bigcup_{i\in I}\le_i,0,1\right).
\]

\begin{theorem}\label{th2}
Let $\mathbf C_1=(C_1,\le,0,1)$ and $\mathbf C_2=(C_2,\le,0,1)$ be bounded chains satisfying $C_1\cap C_2=\{0,1\}$, let $\mathbf L=(L,\vee,\wedge,0,1)$ denote the horizontal sum of $\mathbf C_1$ and $\mathbf C_2$ and let $\,'$ be a unary operation on $L$. Then the following holds:
\begin{enumerate}[{\rm(i)}]
\item The operation $\,'$ is a complementation if and only if $0'=1$, $1'=0$, $(C_1\setminus\{0,1\})'\subseteq C_2\setminus\{0,1\}$ and $(C_2\setminus\{0,1\})'\subseteq C_1\setminus\{0,1\}$,
\item if $\,'$ is a complementation then $\mathbf L$ satisfies identity {\rm(\ref{equ1})},
\item if $\,'$ is a complementation then $\mathbf L$ satisfies De Morgan's laws if and only if $\,'$ is antitone.
\end{enumerate}
\end{theorem}

\begin{proof}
Let $a,b\in L$.
\begin{enumerate}[(i)]
\item First assume $\,'$ to be a complementation. As mentioned in the beginning, we have $0'=1$ and $1'=0$. Moreover, if $a\in C_1\setminus\{0,1\}$ then $a'\in L\setminus C_1=C_2\setminus\{0,1\}$ showing	$(C_1\setminus\{0,1\})'\subseteq C_2\setminus\{0,1\}$. By symmetry, $(C_2\setminus\{0,1\})'\subseteq C_1\setminus\{0,1\}$. Conversely, if $0'=1$, $1'=0$, $(C_1\setminus\{0,1\})'\subseteq C_2\setminus\{0,1\}$ and $(C_2\setminus\{0,1\})'\subseteq C_1\setminus\{0,1\}$ then $\,'$ is a complementation, namely the equalities $a\vee a'=1$ and $a\wedge a'=0$ can be seen by distinguishing the cases $a=0$, $a=1$, $a\in C_1\setminus\{0,1\}$ and $a\in C_2\setminus\{0,1\}$.
\item If $a\in\{0,1\}$ or $b\in\{0,1\}$ then $a+_1b=a+_2b$. If $a,b\in C_1\setminus\{0,1\}$ then $a\vee b\in C_1\setminus\{0,1\}$ and $a'\vee b'\in C_2\setminus\{0,1\}$ and hence
\[
a+_1b=(a'\wedge b)\vee(a\wedge b')=0\vee0=0=(a\vee b)\wedge(a'\vee b')=a+_2b.
\]
If $a\in C_1\setminus\{0,1\}$ and $b\in C_2\setminus\{0,1\}$ then $a'\wedge b\in C_2\setminus\{0,1\}$ and $a\wedge b'\in C_1\setminus\{0,1\}$ and hence
\[
a+_1b=(a'\wedge b)\vee(a\wedge b')=1=1\wedge1=(a\vee b)\wedge(a'\vee b').
\]
The remaining cases can be treated similarly.
\item Assume $\,'$ to be a complementation. If $\mathbf L$ satisfies De Morgan's laws and $a\le b$ then $b'=(a\vee b)'=a'\wedge b'\le a'$. This shows that $\,'$ is antitone. Conversely, assume $\,'$ to be antitone. If $a,b\in C_1$ and $a\le b$ then $(a\vee b)'=b'=b'\wedge a'$. If $a\in C_1\setminus\{0,1\}$ and $b\in C_2\setminus\{0,1\}$ then $(a\vee b)=1'=0=a'\wedge b'$. The other cases can be treated similarly. Hence $\mathbf L$ satisfies the identity $(x\vee y)'\approx x'\wedge y'$. Dual arguments show that $\mathbf L$ satisfies also the identity $(x\wedge y)'\approx x'\vee y'$.
\end{enumerate}
\end{proof}

We apply Theorem~\ref{th2} in the following cases.

\begin{example}\label{ex3}
Theorem~\ref{th2} shows that the non-modular ortholattice $\mathbf O_6=(O_6,\vee,\wedge,{}')$ depicted in Fig.~4
	
\vspace*{-4mm}
	
\begin{center}
\setlength{\unitlength}{7mm}
\begin{picture}(4,8)
\put(2,1){\circle*{.3}}
\put(1,3){\circle*{.3}}
\put(3,3){\circle*{.3}}
\put(1,5){\circle*{.3}}
\put(3,5){\circle*{.3}}
\put(2,7){\circle*{.3}}
\put(2,1){\line(-1,2)1}
\put(2,1){\line(1,2)1}
\put(1,3){\line(0,1)2}
\put(3,3){\line(0,1)2}
\put(2,7){\line(-1,-2)1}
\put(2,7){\line(1,-2)1}
\put(1.85,.3){$0$}
\put(.35,2.85){$a$}
\put(3.4,2.85){$c$}
\put(.35,4.85){$b$}
\put(3.4,4.85){$d$}
\put(1.85,7.4){$1$}
\put(1.2,-.75){{\rm Fig.~4}}
\put(0,-1.75){{\rm Ortholattice $\mathbf O_6$}}
\end{picture}
\end{center}
	
\vspace*{10mm}
	
satisfies identity {\rm(\ref{equ1})} since $\mathbf O_6$ is the horizontal sum of two four-element chains where the antitone involution $\,'$ is defined by
\[
\begin{array}{l|cccccc}
x  & 0 & a & b & c & d & 1 \\
\hline
x' & 1 & d & c & b & a & 0
\end{array}
\]
Consider the lattice $\mathbf O_6$ once more, but define the complementation as follows:
\[
\begin{array}{l|cccccc}
x  & 0 & a & b & c & d & 1 \\
\hline
x' & 1 & c & d & a & b & 0
\end{array}
\]
This operation is again an involution but it is not antitone. The resulting algebra will be denoted by $\mathbf O_6^*$. Again, $\mathbf O_6^*$ satisfies identity {\rm(\ref{equ1})}.

\end{example}

\begin{example}\label{ex2}
The non-modular lattice $\mathbf N_5=(N_5,\vee,\wedge)$ visualized in Fig.~5

\vspace*{-4mm}

\begin{center}
\setlength{\unitlength}{7mm}
\begin{picture}(4,8)
\put(2,1){\circle*{.3}}
\put(1,4){\circle*{.3}}
\put(3,3){\circle*{.3}}
\put(3,5){\circle*{.3}}
\put(2,7){\circle*{.3}}
\put(2,1){\line(-1,3)1}
\put(2,1){\line(1,2)1}
\put(3,3){\line(0,1)2}
\put(2,7){\line(-1,-3)1}
\put(2,7){\line(1,-2)1}
\put(1.85,.3){$0$}
\put(3.4,2.85){$a$}
\put(3.4,4.85){$c$}
\put(.35,3.85){$b$}
\put(1.85,7.4){$1$}
\put(1.2,-.75){{\rm Fig.~5}}
\put(-1,-1.75){{\rm Non-modular lattice $\mathbf N_5$}}
\end{picture}
\end{center}

\vspace*{10mm}

is the horizontal sum of a three-element and a four-element chain. The complementation $\,'$ defined by
\[
\begin{array}{l|ccccc}
x  & 0 & a & b & c & 1 \\
	\hline
x' & 1 & b & a & b & 1
\end{array}
\]
is antitone, but not an involution and it does even not satisfy the identity $x\le x''$ since $c\not\le a=c''$. On the contrary, it satisfies the identity $x''\le x$.

According to Theorem~\ref{th2}, $\mathbf N_5=(N_5,\vee,\wedge,{}')$ satisfies De Morgan's laws and identity {\rm(\ref{equ1})}. One can easily check that $\mathbf N_5$ has just five congruences, namely $\Delta,\mu,\alpha,\beta$ and $\nabla$ defined by
\begin{align*}
\Delta & :=\{(x,x)\mid x\in N_5\}, \\
   \mu & :=\{0\}^2\cup\{a,c\}^2\cup\{b\}^2\cup\{1\}^2, \\
\alpha & :=\{0,b\}^2\cup\{a,c,1\}^2, \\
 \beta & :=\{0,a,c\}^2\cup\{b,1\}^2, \\
\nabla & :=N_5^2.
\end{align*}
The congruence lattice of $\mathbf N_5$ is depicted in Fig.~6:

\vspace*{-4mm}

\begin{center}
\setlength{\unitlength}{7mm}
\begin{picture}(4,8)
\put(2,1){\circle*{.3}}
\put(2,3){\circle*{.3}}
\put(1,5){\circle*{.3}}
\put(3,5){\circle*{.3}}
\put(2,7){\circle*{.3}}
\put(2,1){\line(0,1)2}
\put(2,3){\line(-1,2)1}
\put(2,3){\line(1,2)1}
\put(2,7){\line(-1,-2)1}
\put(2,7){\line(1,-2)1}
\put(1.75,.3){$\Delta$}
\put(2.4,2.85){$\mu$}
\put(.2,4.85){$\alpha$}
\put(3.4,4.85){$\beta$}
\put(1.75,7.4){$\nabla$}
\put(1.2,-.75){{\rm Fig.~6}}
\put(-1.2,-1.75){{\rm Congruence lattice of $\mathbf N_5$}}
\end{picture}
\end{center}

\vspace*{10mm}

and hence $\mathbf N_5$ is subdirectly irreducible. However, we can define the complementation in $\mathbf N_5$ also in a different way, namely as follows:
\[
\begin{array}{l|ccccc}
x  & 0 & a & b & c & 1 \\
\hline
x' & 1 & b & c & b & 1
\end{array}
\]
Then the resulting algebra $\mathbf N_5^*$ satisfies the identities {\rm(\ref{equ1})} and $x\le x''$ since $a\le c=a''=c''$. The ortholattice $\mathbf O_6$ as well as the algebra $\mathbf O_6^*$ has also five congruences, namely $\Delta,\mu,\alpha,\beta$ and $\nabla$ defined by
\begin{align*}
\Delta & :=\{(x,x)\mid x\in O_6\}, \\
   \mu & :=\{0\}^2\cup\{a,b\}^2\cup\{b',a'\}^2\cup\{1\}^2, \\
\alpha & :=\{0,a,b\}^2\cup\{b',a',1\}^2, \\
 \beta & :=\{0,b',a'\}^2\cup\{a,b,1\}^2, \\
\nabla & :=O_6^2.
\end{align*}
The congruence lattices of $\mathbf N_5^*$, $\mathbf O_6$ and $\mathbf O_6^*$ coincide with that of $\mathbf N_5$ and hence also $\mathbf N_5^*$, $\mathbf O_6$ and $\mathbf O_6^*$ are subdirectly irreducible.
\end{example}

The variety $\mathcal W$ of lattices with complementation satisfying identity {\rm(\ref{equ1})} has several interesting subvarieties. Among these are the variety $\mathcal B$ of Boolean algebras and the varieties $\mathcal V(\mathbf N_5)$, $\mathcal V(\mathbf N_5^*)$, $\mathcal V(\mathbf O_6)$ and $\mathcal V(\mathbf O_6^*)$ generated by the corresponding algebras. These do not coincide because they satisfy different identities despite the fact that as lattices (without complementation) $\mathbf N_5$ and $\mathbf N_5^*$ as well as $\mathbf O_6$ and $\mathbf O_6^*$ are isomorphic.

It is worth noticing that also the ortholattice visualized in Fig.~3 is subdirectly irreducible and hence it also generates a subvariety of the variety $\mathcal W$ which does not coincide with afore mentioned ones.

By J\'onsson Lemma, the two-element Boolean lattice $\mathbf2=\{0,1\}$ is the only subdirectly irreducible member of $\mathcal B$, $\mathbf2$ and $\mathbf N_5$ are the only subdirectly irreducible members of $\mathcal V(\mathbf N_5)$, and $\mathbf2$ and $\mathbf O_6$ are the only subdirectly irreducible members of $\mathcal V(\mathbf O_6)$.

Moreover, $\mathcal V(\mathbf O_6)$ satisfies the identity $x''\approx x$ since the complementation in $\mathbf 2$ as well as in $\mathbf O_6$ is an involution. On the other hand, $\mathcal V(\mathbf N_5)$ does not satisfy this identity, it satisfies only the identities $x'''\approx x'$ and $x''\le x$.

That the varieties $\mathcal V(\mathbf N_5)$ and $\mathcal V(\mathbf O_6)$ both satisfy De Morgan's laws follows from Theorem~\ref{th2}.

As pointed out already by G.~Birkhoff \cite{B35}, if $\mathbf A$ is a subdirectly irreducible algebra and $\mathcal V(\mathbf A)$ denotes the variety generated by $\mathbf A$ then the cardinality of the free algebra $\mathbf F_{\mathcal V(\mathbf A)}(n)$ in $\mathcal V(\mathbf A)$ with $n$ free generators can be estimated as follows:
\begin{equation}\label{equ9}
|\mathbf F_{\mathcal V(\mathbf A)}(n)|\le|A|^{|A|^n}.
\end{equation}
However, J.~Berman \cite B used the estimation $|\mathbf F_{\mathcal V(\mathbf A)}(n)|\le|A|^k$ and got a method how to determine a ``good'' number $k$. The process may not give a minimal $k$ in general, though in congruence-distributive varieties it should. E.~W.~Kiss and R.~Freese developed a computer program accessible at the address
\[
http:\slash\slash www.cs.elte.hu\slash\,\tilde{}ewkiss\slash software\slash naprog\slash
\]
computing the optimal estimation of $k$. Of course, the varieties $\mathcal V(\mathbf N_5)$ and $\mathcal V(\mathbf O_6)$ are locally finite and the cardinalities of the corresponding free algebras and estimations for k computed by this program are as follows:

For $\mathcal V(\mathbf N_5)$ we have
\[
\begin{array}{r|r|r}
n & |\mathbf F_{\mathcal V(\mathbf N_5)}(n)| &    k \\
\hline
1 &                                        5 &    1 \\
2 &                                      152 &    9 \\
3 &                                \ge100036 &   61 \\
4 &                                          &  369 \\
5 &                                          & 2101
\end{array}
\]	
and for $\mathcal V(\mathbf O_6)$
\[
\begin{array}{r|r|r}
n & |\mathbf F_{\mathcal V(\mathbf O_6)}(n)| &    k \\
\hline
1 &                                        4 &    2 \\
2 &                                      100 &    4 \\
3 &                                \ge249275 &   48 \\
4 &                                          &  400 \\
5 &                                          & 2880
\end{array}
\]
The same tables hold for $\mathbf N_5^*$ and $\mathbf O_6^*$, respectively.

\section{Identities of symmetric differences}
	
As mentioned in the introduction, it is well-known (see e.g.\ \cite B) that in a Boolean lattice $(L,\vee,\wedge,{}',0,1)$ the term operations $+_1$ and $+_2$ coincide and these operations are associative, i.e.\ they satisfy the identities
\[
(x+_1y)+_1z\approx x+_1(y+_1z)\text{ and }(x+_2y)+_2z\approx x+_2(y+_2z),
\]
respectively. The question arises whether, conversely, in a lattice with complementation, associativity of $+_1$ or $+_2$ implies lattice distributivity. We will show that this is indeed the case.

At first we prove an auxiliary result using identites similar to those studied in \cite{CLP} and \cite{CP}.

\begin{lemma}\label{lem3}
Let $(L,\vee,\wedge,{}',0,1)$ be a lattice with complementation. Then the following are equivalent:
\begin{enumerate}[{\rm(i)}]
\item $x\wedge y\approx x\wedge(x'\vee y)$,
\item $x\vee y\approx x\vee(x'\wedge y)$.
\end{enumerate}
\end{lemma}

\begin{proof}
Assume (i). Then, by Lemma~\ref{lem2} we have $x''\approx x$, and from (i) we get
\begin{equation}\label{equ13}
x'\wedge y\approx x'\wedge(x''\vee y)\approx x'\wedge(x\vee y).
\end{equation}
Using this we obtain
\[
(x\vee y)'\wedge x'\approx(x\vee y)'\wedge\big((x\vee y)\vee x'\big)\approx(x\vee y)'\wedge1\approx(x\vee y)'.
\]
Thus
\begin{equation}\label{equ14}
(x\vee y)'\wedge x'\approx(x\vee y)'
\end{equation}
and hence
\[
x\vee(x'\vee y)'\approx x\vee\big((x'\vee y)'\wedge x''\big)\approx x\vee\big((x'\vee y)'\wedge x\big)\approx x
\]
showing
\begin{equation}\label{equ15}
x\approx x\vee(x'\vee y)'.
\end{equation}
Using (i) we get
\[
x\wedge\big(y\vee(y\vee x')'\big)\approx x\wedge\big(x'\vee y\vee(y\vee x')'\big)\approx x\wedge1\approx x
\]
and hence
\[
x\approx x\wedge\big(y\vee(y\vee x')'\big).
\]
Using this and (\ref{equ15}) we obtain
\[
x\vee(x\vee y')'\approx\big(x\vee(x\vee y')'\big)\vee\Big(\big(x\vee(x\vee y')'\big)\wedge y\Big)\approx\big(x\vee(x\vee y')'\big)\vee y\approx x\vee\big((x\vee y')'\vee y\big)\approx x\vee y.
\]
Thus
\begin{equation}\label{equ17}
x\vee y\approx x\vee(x\vee y')'.
\end{equation}
Using this and (\ref{equ13}) we get
\[
(x\vee y)\wedge x'\approx\big(x\vee(x\vee y')'\big)\wedge x'\approx(x\vee y')'\wedge x'.
\]
Thus
\[
(x\vee y)\wedge x'\approx(x\vee y')'\wedge x'.
\]
From this, (\ref{equ14}) and (\ref{equ13}) we obtain
\[
(x\vee y)'\approx(x\vee y)'\wedge x'\approx(x\vee y'')'\wedge x'\approx(x\vee y')\wedge x'\approx x'\wedge y'.
\]
Thus
\[
(x\vee y)'\approx x'\wedge y'.
\]
Finally, using this and (\ref{equ17}) we get
\[
x\vee y\approx x\vee(x\vee y')'\approx x\vee(x'\wedge y'')\approx x\vee(x'\wedge y).
\]
Therefore (ii) holds. On the other hand, if (ii) holds, then again $x''\approx x$ by Lemma~\ref{lem2}. The proof that (ii) implies (i) follows by dualizing the preceding arguments.
\end{proof}

Now we get a connection of the afore mentioned identities with distributivity of the lattice in question. This is an essential result for the proof of our next theorem.

\begin{theorem}\label{th3}
Let $\mathbf L=(L,\vee,\wedge,{}',0,1)$ be a lattice with complementation satisfying one of the following identities
\begin{enumerate}[{\rm(i)}]
\item $x\wedge y\approx x\wedge(x'\vee y)$,
\item $x\vee y\approx x\vee(x'\wedge y)$.
\end{enumerate}
Then $\mathbf L$ is distributive and hence Boolean.
\end{theorem}

\begin{proof}
According to Lemmas~\ref{lem2} and \ref{lem3} we have $x''\approx x$, and (i) and (ii) are equivalent. Therefore we may assume that both identities are satisfied. Using these identities we get
\[
x\wedge(y\vee z)\approx x\wedge(x'\vee y\vee z)\approx x\wedge\big(x'\vee(x\wedge y)\vee z\big)\approx x\wedge\big((x\wedge y)\vee z\big).
\]
Thus
\begin{equation}\label{equ20}
x\wedge(y\vee z)\approx x\wedge\big((x\wedge y)\vee z\big).
\end{equation}
Using this we obtain
\begin{align*}
x\wedge(y\vee z) & \approx\big(x\wedge(y\vee z)\big)\vee\big(x\wedge(y\vee z)\wedge y\big)\approx \\
                 & \approx\big(x\wedge(y\vee z)\big)\vee\Big(x\wedge\big((x\wedge y)\vee z\big)\wedge y\Big)\approx \\
                 & \approx\big(x\wedge(y\vee z)\big)\vee\Big((x\wedge y)\wedge\big((x\wedge y)\vee z\big)\Big)\approx\big(x\wedge(y\vee z)\big)\vee(x\wedge y)
\end{align*}
and hence
\[
x\wedge(y\vee z)\approx\big(x\wedge(y\vee z)\big)\vee(x\wedge y).
\]
Finally, using this, (ii), (\ref{equ20}) and (i) we get
\begin{align*}
x\wedge (y\vee z) & \approx\big(x\wedge(y\vee z)\big)\vee(x\wedge y)\approx(x\wedge y)\vee\Big((x\wedge y)'\wedge\big(x\wedge(y\vee z)\big)\Big)\approx \\
                  & \approx(x\wedge y)\vee\Big((x\wedge y)'\wedge x\wedge\big((x\wedge y)\vee z\big)\Big)\approx \\
                  & \approx(x\wedge y)\vee\Big((x\wedge y)'\wedge\big((x\wedge y)\vee z\big)\wedge x\Big)\approx(x\wedge y)\vee\big((x\wedge y)'\wedge z\wedge x\big)\approx \\
                  & \approx(x\wedge y)\vee(x\wedge z).
\end{align*}
\end{proof}

Our main result concerning the associativity of the symmetric difference is as follows.

\begin{theorem}\label{th5}
Let $\mathbf L=(L,\vee,\wedge,{}',0,1)$ be a lattice with complementation. Then $\mathbf L$ is Boolean if and only if either $+_1$ or $+_2$ is associative.
\end{theorem}

\begin{proof}
If $\mathbf L$ is Boolean then $+_1=+_2$ and this operation is associative (see \cite{B79}). Conversely, assume either $+_1$ or $+_2$ to be associative. Let $+\in\{+_1,+_2\}$. Then
\[
x+0\approx x, x+1\approx x'\text{ and }x+x\approx0
\]
and hence
\[
(x+y)'\approx x+(y+1)\approx x+y'\approx x'+y\text{ and }x''\approx x+1+1\approx x+0\approx x.
\]
Thus
\[
(x+y)'\approx x+y'\approx x'+y\text{ and }x''\approx x.
\]
For $+_1$ we have
\begin{equation}\label{equ24}
x\approx(x+_1y)+_1y\approx\big(y\wedge(x+_1y)'\big)\vee\big(y'\wedge(x+_1y)\big).
\end{equation}
Using this we obtain
\begin{align*}
x\wedge(x+_1y)' & \approx x\wedge(x+_1y)'\wedge\Big(\big(x\wedge(x+_1y)'\big)\vee\big(x'\wedge(x+_1y)\big)\Big)\approx x\wedge(x+_1y)'\wedge y\approx \\
                & \approx(x\wedge y)\wedge(x'+_1y)\approx(x\wedge y)\wedge\big((x\wedge y)\vee(x'\wedge y')\big)\approx x\wedge y.
\end{align*}
Thus
\[
x\wedge(x+_1y)'\approx x\wedge y.
\]
From this and (\ref{equ24}) we get
\begin{align*}
x & \approx\big(y\wedge(x+_1y)'\big)\vee\big(y'\wedge(x+_1y)\big)\approx(x\wedge y)\vee\big(y'\wedge(x+_1y'')\big)\approx \\
  & \approx(x\wedge y)\vee\big(y'\wedge(x+_1y')'\big)\approx(x\wedge y)\vee(x\wedge y').
\end{align*}
Thus
\[
x\approx(x\wedge y)\vee(x\wedge y').
\]
Using this we obtain
\[
x'\vee y\approx(x'\wedge y)\vee(x'\wedge y')\vee(x\wedge y)\vee(x'\wedge y)\approx(x'\wedge y)\vee(x'\wedge y')\vee(x\wedge y)\approx x'\vee(x\wedge y).	
\]
Thus $x'\vee y\approx x'\vee(x\wedge y)$ and, consequently, $x\vee y\approx x\vee(x'\wedge y)$. By Theorem~\ref{th3} this implies that $\mathbf L$ is Boolean. For $+_2$ we can use similar arguments. We have
\begin{equation}\label{equ27}
x\approx(x+_2y)+_2y\approx\big(y\vee(x+_2y)\big)\wedge\big(y'\vee(x+_2y)'\big).
\end{equation}	
Using this we get
\begin{align*}
x\vee(x+_2y) & \approx\big(x\vee(x+_2y)\big)\vee\Big(\big(x\vee(x+_2y)\big)\wedge\big(x'\vee(x+_2y)'\big)\Big)\approx x\vee(x+_2y)\vee y\approx \\
& \approx(x\vee y)\wedge(x+_2y)\approx(x\vee y)\vee\big((x\vee y)\wedge(x'\vee y')\big)\approx x\vee y.	
\end{align*}
Thus
\[
x\vee(x+_2y)\approx x\vee y.
\]
From this and (\ref{equ27}) we obtain
\[
x\approx\big(y\vee(x+_2y)\big)\wedge\big(y'\vee(x+_2y)'\big)\approx(x\vee y)\wedge\big(y'\vee(x+_2 y')\big)\approx(x\vee y)\wedge(x\vee y').
\]
Thus
\[
x\approx(x\vee y)\wedge(x\vee y')
\]
and hence
\[
x\wedge y\approx(x\vee y)\wedge(x\vee y')\wedge(x\vee y)\wedge(x'\vee y)\approx(x\vee y)\wedge(x\vee y')\wedge(x'\vee y)\approx x\wedge(x'\vee y).
\]
This shows $x\wedge y\approx x\wedge(x'\vee y)$ which, by Theorem~\ref{th3}, implies that $\mathbf L$ is Boolean.
\end{proof}

As shown above, associativity of the symmetric difference in a lattice $\mathbf L$ with complementation yields distributivity of $\mathbf L$. Surprisingly, also a simple identity in only two variables formulated for the symmetric difference yields the same, see the following result.

\begin{theorem}
Let $\mathbf L=(L,\vee,\wedge,{}',0,1)$ be a lattice with complementation satisfying identity {\rm(\ref{equ1})}. Then $\mathbf L$ is a Boolean algebra if and only if $+$ satisfies the identity
\[
(x+y)+y\approx x
\]
where either $+=+_1$ or $+=+_2$.
\end{theorem}

\begin{proof}
Of course, if $\mathbf L$ is a Boolean algebra then according to Theorem~\ref{th5} the symmetric difference is associative and hence it satisfies $(x+y)+y\approx x+(y+y)\approx x+0\approx x$. Conversely, assume $\mathbf L$ to satisfy the identity $(x+y)+y\approx x$. If $\mathbf L$ satisfies the identity $(x+_1y)+_1y\approx x$ then
\begin{align*}
	x'\wedge y & \approx(x'\wedge y)\wedge\big((x'\wedge y)\vee(x\wedge y')\big)\approx x'\wedge y\wedge(x+_1y)\approx \\
	& \approx x'\wedge\big((y+_1x)+_1x\big)\wedge(x+_1y)\approx \\
	& \approx\big( x'\wedge(x+_1y)\big)\wedge\Big(\big((y+_1x)'\wedge x\big)\vee\big((y+_1x)\wedge x'\big)\Big)\approx x'\wedge(x+_1y)	
\end{align*}
and hence
\[
x\vee y\approx x\vee\big((y+_1x)+_1x\big)\approx x\vee\big((y+_1x)'\wedge x\big)\vee\big((y+_1x)\wedge x'\big)\approx x\vee\big(x'\wedge(x+_1y)\big)\approx x\vee(x'\wedge y)
\]
which implies that $\mathbf L$ is Boolean according to Theorem~\ref{th3}. If $\mathbf L$ satisfies the identity $(x+_2y)+_2y\approx x$ then
\begin{align*}
	x\vee y & \approx(x\vee y)\vee\big((x\vee y)\wedge(x'\vee y')\big)\approx(x\vee y)\vee(x+_2y)\approx \\
	& \approx x\vee\big((y+_2x)+_2x\big)\vee(x+_2y)\approx \\
	& \approx x\vee(x+_2y)\vee\Big(\big(y+_2x)\vee x\big)\wedge\big((y+_2x)'\vee x'\big)\Big)\approx x\vee(x+_2y)	
\end{align*}
and hence
\begin{align*}
	x'\wedge y & \approx x'\wedge\big((y+_2x)+_2x\big)\approx x'\wedge\big((y+_2x)\vee x\big)\wedge\big((y+_2x)'\vee x'\big)\approx \\
	& \approx x'\wedge\big(x\vee(x+_2y)\big)\approx x'\wedge(x\vee y)
\end{align*}
which because of
\begin{align*}
	x' & \approx1\wedge(x'\vee0)\approx(x\vee1)\wedge(x'\vee1')\approx x+_21, \\
	x'' & \approx(x+_21)+_21\approx x
\end{align*}
implies $x\wedge y\approx x''\wedge y\approx x''\wedge(x'\vee y)\approx x\vee(x'\vee y)$ proving $\mathbf L$ to be Boolean according to Theorem~\ref{th3}.
\end{proof}

In \cite{CP} it was shown that a lattice $(L,\vee,\wedge,{}')$ with a unary operation $\,'$ is Boolean if and only if it satisfies the identities
\[
x'\vee(x\wedge y)\approx x'\vee y\text{ and }x'\wedge(x\vee y)\approx x'\wedge y.
\]
In \cite{CLP} single identities are presented that force a lattice with a unary operation to be Boolean. Finally, in \cite{PR} it was proved that a lattice $(L,\vee,\wedge,{}')$ with a unary operation $\,'$ is Boolean if and only if it satisfies the identity
\[
(x\wedge y)\vee(x\wedge y')\approx(x\vee y)\wedge(x\vee y')
\]
(Proposition~4.2.1).

\section{De Morgan's laws}

It is well-known that a lattice $(L,\vee,\wedge,{}')$ with an antitone involution $\,'$ satisfies De Morgan's laws
\[
(x\vee y)'\approx x'\wedge y'\text{ and }(x\wedge y)'\approx x'\wedge y'.
\]
However, De Morgan's laws may be satisfied also in the case that $\,'$ is not an involution.

Now we are interested in conditions under which a lattice with a unary operation $\,'$ satisfies De Morgan's laws. At first, we see that the lattice $\mathbf N_5$ with the complementation defined in Example~\ref{ex2} satisfies De Morgan's laws despite the fact that the complementation $\,'$ is not an involution.

\begin{example}\label{ex1}
The complementation $\,'$ of the non-modular lattices $\mathbf N_5$ and $\mathbf N_5^*$ is antitone, but not an involution, but these lattices still satisfy De Morgan's laws. The complementation $\,'$ of the non-modular lattice $\mathbf O_5$ is antitone and an involution and satisfies De Morgan's laws. The complementation $\,'$ of the non-modular lattice $\mathbf O_6^*$ is not antitone, but it is an involution, and does not satisfy De Morgan's laws.
\end{example}

\begin{theorem}\label{th4}
Let $\mathbf L=(L,\vee,\wedge,{}')$ be a lattice with a unary operation $\,'$. Then the following are equivalent:
\begin{enumerate}[{\rm(i)}]
\item $\mathbf L$ satisfies the identities $(x\vee y)'\approx x'\wedge y'$ and $(x\wedge y)'\approx x'\vee y'$,
\item the operation $\,'$ is antitone and $\mathbf L$ satisfies the identities $(x\vee y)'\wedge x'\approx x'\wedge y'$ and $(x\wedge y)'\vee x'\approx x'\vee y'$.
\end{enumerate}
\end{theorem}

\begin{proof}
We prove that the following are equivalent:
\begin{enumerate}
\item[(iii)] $\mathbf L$ satisfies the identity $(x\vee y)'\approx x'\wedge y'$,
\item[(iv)] the operation $\,'$ is antitone and $\mathbf L$ satisfies the identity $(x\vee y)'\wedge x'\approx x'\wedge y'$.
\end{enumerate}
The rest follows by duality. Let $a,b\in L$. \\
(iii) $\Rightarrow$ (iv): \\
If $a\le b$ then $b'=(a\vee b)'=a'\wedge b'\le a'$ showing that $\,'$ is antitone. Moreover, $(a\vee b)'\wedge a'=a'\wedge b'\wedge a'=a'\wedge b'$ proving (iv). \\
(iv) $\Rightarrow$ (iii): \\
Because of $a\le a\vee b$ and $b\le a\vee b$ we have $(a\vee b)'\le a'$ and $(a\vee b)'\le b'$ and hence $(a\vee b)'\le a'\wedge b'$. Conversely, (iv) implies $a'\wedge b'=(a\vee b)'\wedge a'\le(a\vee b)'$ which together with $(a\vee b)'\le a'\wedge b'$ yields $(a\vee b)'=a'\wedge b'$.
\end{proof}

\begin{example}
One can easily check that condition {\rm(ii)} of Theorem~\ref{th4} is satisfied in the lattices $\mathbf N_5$, $\mathbf N_5^*$ and $\mathbf O_6$, but not in the lattice $\mathbf O_6^*$.
\end{example}








Authors' addresses:

Ivan Chajda \\
Palack\'y University Olomouc \\
Faculty of Science \\
Department of Algebra and Geometry \\
17.\ listopadu 12 \\
771 46 Olomouc \\
Czech Republic \\
ivan.chajda@upol.cz

Helmut L\"anger \\
TU Wien \\
Faculty of Mathematics and Geoinformation \\
Institute of Discrete Mathematics and Geometry \\
Wiedner Hauptstra\ss e 8-10 \\
1040 Vienna \\
Austria, and \\
Palack\'y University Olomouc \\
Faculty of Science \\
Department of Algebra and Geometry \\
17.\ listopadu 12 \\
771 46 Olomouc \\
Czech Republic \\
helmut.laenger@tuwien.ac.at
\end{document}